\documentclass[12pt]{article}

\usepackage{amssymb}
\usepackage{amsmath}
\usepackage{mathdots}
\usepackage{amsbsy}
\usepackage{amscd}
\usepackage{amsthm}
\usepackage{mathrsfs}
\usepackage{verbatim}
\usepackage[colorlinks]{hyperref}
\usepackage{fullpage}

\usepackage{url}

\usepackage[english]{babel}
\usepackage[T1]{fontenc}
\usepackage[utf8]{inputenc}
\usepackage{tikz}
\usepackage{authblk}

\theoremstyle{definition}
\newtheorem{Theorem}{Theorem}[section]
\newtheorem{Lemma}[Theorem]{Lemma}

\newtheorem{Remark}[Theorem]{Remark}

\newtheorem{Proposition}[Theorem]{Proposition}

\newcommand{\R}{\mathbb{R}}

\newcommand{\Z}{\mathbb{Z}}

\newcommand{\mt}[1]{\text{#1}}

\begin{document}

\title{The Cross-section of a Spherical Double Cone}
\author{Mahir Bilen Can}
\affil{{\small mahirbilencan@gmail.com}}
\normalsize

\date{\today}
\maketitle

\begin{abstract}
We show that the poset of $SL(n)$-orbit closures in the product of two partial flag varieties 
is a lattice if the action of $SL(n)$ is spherical. 
\vspace{.5cm}

\end{abstract}

\maketitle

\section{Introduction}\label{S:Introduction}

Let $G$ be a connected reductive algebraic group.
A normal algebraic variety $X$ is called a spherical 
$G$-variety if there exists an algebraic action 
$G\times X\rightarrow X$ such that the restriction 
of the action to a Borel subgroup $B$ of $G$ has 
an open orbit in $X$. In this case, we say that 
the action is spherical. 

\vspace{.25cm}

Let $P_1,\dots, P_k \subset G$ be a list of parabolic 
subgroups containing the same Borel subgroup $B$ and let 
$X$ denote the product variety $X=G/P_1\times \cdots \times G/P_k$.
Then $X$ is a smooth, hence normal, $G$-variety via 
the diagonal action. 
The study of functions on an affine 
cone over $X$ is important for understanding the 
decompositions of tensor products of 
representations of $G$, see~\cite{PopovVinberg,Panyushev99}.
In particular, determining when the diagonal action of $G$ on $X$ is 
spherical is important for understanding the multiplicity-free
representations of $G$.
In his ground breaking article~\cite{Littelmann}, 
Littelmann initiated the classification problem 
and gave a list of all possible pairs of maximal parabolic 
subgroups $P_1,P_2$ such that $G/P_1\times G/P_2$ 
is a spherical $G$-variety.
In~\cite{MWZ1}, for group $G=SL(n)$ and in~\cite{MWZ2} 
for $G=Sp(2n)$, Magyar, Weyman, and Zelevinski classified
the parabolic subgroups $P_1,\dots, P_k$ such that the 
product $X=G/P_1\times \cdots \times G/P_k$
is a spherical $G$-variety. According to~\cite{MWZ1}, 
if $X$ is a spherical $G$-variety, then the number of factors is at most 3, 
and $k=3$ occurs in only special cases. 
Therefore, the gist of the problem lies in the case $k=2$. 
This case is settled in full detail by Stembridge. In~\cite{Stembridge},
for a semisimple complex algebraic group $G$, 
Stembridge listed all pairs of parabolic 
subgroups $(P_1,P_2)$ such that $G/P_1\times G/P_2$ is a 
spherical $G$-variety.

\vspace{.5cm}

For motivational purposes, we will mention some recent related developments. 
Let $K$ be a connected reductive subgroup of $G$ and let 
$P$ be a parabolic subgroup of $G$. One of the major
open problems in the classification of spherical actions is the following: 
What are the possible triplets $(G,K,P)$ 
such that $G/P$ is a spherical $K$-variety? 
When $K$ is a Levi subgroup of a parabolic 
subgroup $Q$, this question is equivalent to 
asking when $G/Q\times G/P$ is a spherical $G$-variety
via diagonal action; it has a known solution as we mentioned earlier.
For an explanation of this equivalence, see~\cite[Lemma 5.4]{AP}. 
In~\cite{AP}, Avdeev and Petukhov gave a complete 
answer to the above question in the case $G=SL(n)$.  
If we assume that $K$ is a symmetric subgroup of $G$, 
then our initial question is equivalent to 
asking when $G/P\times K/B_K$ has an open $K$-orbit
via its diagonal action. Here, $B_K$ is a Borel subgroup of $K$. 
In this case, the answer is recorded in~\cite{Heetal}. 
See also the related work of Pruijssen~\cite{VanPruijssen}.
Finally, let us mention another extreme situation where 
the answer is known: 
$G$ is an exceptional simple group, 
$P$ is a maximal parabolic subgroup, and 
$K$ is a maximal reductive subgroup of $G$,
see~\cite{Niemann}.

\vspace{.5cm}

We go back to the products of flag varieties and let 
$P$ and $Q$ be two parabolic subgroups from $G$. 
From now on we will call a product variety of the form 
$G/P\times G/Q$ a double flag variety. If the diagonal 
action of $G$ on a double flag variety 
$X=G/P\times G/Q$ is spherical, then we will call 
$X$ a spherical double flag variety for $G$. 
As it is shown by Littelmann in his previously
mentioned article, 
the problem of deciding if a double flag variety is 
spherical or not is closely related to a study of the 
invariants of a maximal unipotent subgroup in the 
coordinate ring of an affine cone over $X$.
In turn, this study is closely related to the 
combinatorics of the $G$-orbits in $X$. 
In this regard, our goal in this note is to 
prove the following result on the poset of 
inclusion relationships between the $G$-orbit closures
in a spherical double flag variety.

\begin{Theorem}\label{T:main}
Let $G$ denote the special linear group $SL(n+1)$.
If $X$ is a spherical double flag variety for $G$, then the 
poset of $G$-orbit closures in $X$ is a lattice. 
\end{Theorem}
In fact, we have a precise 
description of the possible lattices in Theorem~\ref{T:main}.
It turns out that the Hasse diagram of such a lattice look like 
a ``ladder'', or the lattice is a chain, see Theorem~\ref{T:Type A G orbits}.

\vspace{.5cm}

The structure of our paper is as follows. 
In the next section we set up our notation 
and review some basic facts about 
the double-cosets of parabolic subgroups. 
In Subsection~\ref{SS:2}, we show that 
the inclusion poset of $G$-orbit closures in 
$G/P\times G/Q$ is isomorphic to the inclusion poset 
of $P$-orbit closures in $G/Q$. 
In Subsection~\ref{SS:4} we review
the concept of tight Bruhat order
due to Stembridge. We use the information
gained from this subsection in our analysis 
of the cases that are considered in the subsequent 
Section~\ref{S:main}, where we prove 
our main result.

\section{Preliminaries}\label{S:Preliminaries}

\subsection{}\label{SS:1}

For simplicity, let us assume that $G$ is a
semisimple simply-connected complex 
algebraic group and let $B$ be a Borel subgroup of $G$.
Let $T$ be a maximal torus of $G$ that is contained in $B$.
The unipotent radical of $B$ is denoted by $U$, so that 
$B=UT$. We denote by $\Phi$ the root system corresponding
to the pair $(G,T)$ and we denote by $\Delta$ the subset of 
simple roots corresponding to $B$. 
A parabolic subgroup $P$ of $G$ is said to be standard
with respect to $B$ if the inclusion $B\subseteq P$ holds
true. In this case, $P$ is uniquely determined by a 
subset $I \subseteq \Delta$.

\vspace{.25cm}

The Weyl group of $(G,T)$, that is $N_G(T)/T$, 
is denoted by $W$
and we use the letter $R$ to denote the 
set of simple reflections $s_\alpha \in W$, 
where $\alpha \in \Delta$. 
We will allow ourselves be confused
by using the letter $I$ (and $J$) to denote a 
subset of $\Delta$ or the corresponding
subset of simple reflections in $R$. 
The {\em length} of an element $w\in W$,
denoted by $\ell(w)$, is the 
minimal number of Coxeter generators
$s_i\in R$ that is needed for the equality  
$w = s_1\cdots s_k$ hold true.
In this case, the product $s_1\cdots s_k$
is called a reduced expression for $w$. 
Note that when $W$ is the symmetric group
of permutations $S_{n+1}$, the length
of a permutation $w=w_1\dots w_{n+1}\in S_{n+1}$ is 
equal to the number of pairs $(i,j)$ with 
$1\leq i < j \leq n+1$ and $w_i > w_j$.

\vspace{.25cm}

The Bruhat-Chevalley order on $W$ can be defined 
by declaring $v\leq w$ $(w,v\in W)$ 
if a reduced expression of $v$ is obtained
from a reduced expression $s_1\cdots s_k$ of $w$ by
deleting some of the simple reflections $s_i$.

\vspace{.25cm}

Let $X(T):=\mt{Hom}(T,\mathbb{G}_m)$ denote the group of characters
of the maximal torus $T$. 
Let $\{ \omega_1,\dots, \omega_r\}$ 
denote the set of fundamental 
weights corresponding to the set of simple roots 
$\Delta =\{\alpha_1,\dots, \alpha_r \}$. 
By our assumptions on $G$, we have $\omega_i \in X(T)$ for every $i\in \{1,\dots, r\}$.
The Weyl group $W$ acts on the weight lattice, that is $X(T)$. 
Let $\mathbf{E}$ denote the real vector 
space that is spanned by the fundamental
weights, so that $\mathbf{E} = X(T)\otimes_\Z \R$.
The action of $W$ on the weight lattice extends to give 
a linear action on $\mathbf{E}$.
A vector $\theta$ from $\mathbf{E}$
is called a dominant vector if it is of the form
$\theta = a_1 \omega_1 +\cdots + a_r \omega_r$,
where $a_i$'s are nonnegative real numbers. 
Let $W(\omega_i)$ ($i\in \{1,\dots, r\}$)
denote the isotropy group of $\omega_i$ in $W$.
Then $W(\omega_i)$ ($i\in \{1,\dots, r\}$) is a parabolic subgroup of $W$,
and furthermore, the subgroup of $G$ that is generated by $B$ 
and $W(\omega_i)$ is a maximal parabolic subgroup.

\subsection{}\label{SS:2}

Let $G$ act on two irreducible varieties 
$X_1$ and $X_2$, and let 
$x_i \in X_i$, $i=1,2$ be two points 
in general positions. If $G_i \subset G$ denotes 
the stabilizer subgroup of $x_i$ in $G$, then 
$\mt{Stab}_G(x_1\times x_2)$ coincides with 
the stabilizer in $G_1$ of a point in general 
position from $G/G_2$ (or, equivalently, with 
the stabilizer in $G_2$ of a point in general position from $G/G_1$),
see~\cite{Panyushev99}.

\vspace{.25cm}

Let $P_1$ and $P_2$ be two parabolic subgroups of $G$. 
By applying the idea from the previous paragraph
to the double flag variety $X:= G/P_1\times G/P_2$,
where $B\subset P_1\cap P_2$, 
we notice that the study of $G$-orbits in $X$ 
reduces to the study of $P_1$-orbits in the flag
variety $G/P_2$. 
But more is true; 
this correspondence between $G$-orbits 
and $P_1$-orbits respects the inclusions
of their closures in Zariski topology.

\begin{Lemma}\label{L:poset isom}
The poset of $G$-orbit closures in $X$ is 
isomorphic to the poset of $P_2$-orbit closures
in $G/P_1$.
\end{Lemma}

\begin{proof}
Let $X$ denote $G/P_1 \times G/P_2$. 
The canonical projection $\pi : X\rightarrow G/P_2$
is $G$-equivariant and it turns $X$ into a 
homogenous fiber bundle over $G/P_2$ with fiber $G/P_1$
at every point $gP_2$ ($g\in G$)
of the base $G/P_2$. 
To distinguish it from the other fibers,
let us denote by $Y$ the fiber $G/P_1$ 
at the `origin' $eP_2$ of $G/P_2$. Then 
any $G$-orbit in $X$ meets $Y$.
Note also that if $g\cdot y \in Y$
for some $g\in G$ and $y\in X$, then 
$g\in P_2$. 
There are two useful consequences 
of this observation;
1) $Y$ is a $P_2$-variety;  
2) any $G$-orbit 
$O$ meeting $Y$, actually meets $Y$ 
along a $P_2$-orbit. 
Therefore, the map $O\mapsto O\cap Y$
gives a bijection between the set
of all $G$-orbits in $X$ and the set 
of all $P_2$-orbits in $Y$. 

Since $G$ and $P_2$ are connected
algebraic groups, the Zariski closures 
of their orbits are irreducible. Furthermore, 
the boundaries of the orbit closures 
are unions of orbits of smaller dimensions. 
At the same time, $Y$ is closed in $X$, 
therefore, the extension of the 
orbit-correspondence 
map,
\begin{align}
\overline{O} \longmapsto \overline{O}\cap Y,
\end{align}
gives a poset isomorphism between 
the inclusion orders on the Zariski 
closures of $G$-orbits in $X$ 
and the Zariski closures of $P_2$-orbits in $Y$.
This finishes the proof of our lemma.
\end{proof}

\begin{Remark}
By looking at the $(P_1,P_2)$-double cosets
in $G$, as far as the combinatorics of 
orbit closures is concerned, 
we see that there is no real difference
between the study of $P_1$-orbits in $G/P_2$
and the study of $P_2$-orbits in $G/P_1$. 
\end{Remark}

\subsection{}\label{SS:3}

We preserve our assumptions/notation from the previous 
subsections; $P_1$ and $P_2$ are two 
standard parabolic subgroups with respect to $B$. 
If $I$ and $J$ are the subsets of $R$ (or, of $\Delta$) 
that determine $P_1$ and $P_2$, respectively, 
then we will write $P_I$ (resp. $P_J$) in place
of $P_1$ (resp. $P_2$).
The Weyl groups of $P_I$ and $P_J$ 
will be denoted by $W_I$ and $W_J$, respectively. 
In this subsection, we will present some 
well-known facts regarding the set of 
$(W_I,W_J)$-double cosets in $W$, 
denoted by $W_I \backslash W / W_J$.

\vspace{.25cm}

First of all, the set $W_I \backslash W / W_J$
is in a bijection with the set of $P_I$-orbits in $G/P_J$,
see~\cite[Section 21.16]{Borel}.
For $w\in W$, we denote by $[w]$ 
the double coset $W_I w W_J$. 
Let 
$$
\pi : W \rightarrow W_1\backslash W / W_2
$$ 
denote the canonical projection onto the set of 
$(W_1,W_2)$-double cosets. 
It turns out that the preimage in $W$ of every double coset 
in $W_1\backslash W / W_2$ is an interval with 
respect to Bruhat-Chevalley order, 
hence it has a unique maximal and a 
unique minimal element, see~\cite{Curtis85}.
Moreover, if $[w], [w']\in W_1\backslash W / W_2$ 
are two double cosets, $w_1$ and $w_2$ are the maximal 
elements of $[w]$ and $[w']$, respectively, 
then 
$w \leq w'$ if and only if $w_1 \leq w_2$,
see~\cite{HohlwegSkandera}.
It follows that $W_1\backslash W / W_2$ has 
a natural combinatorial partial ordering 
defined by 
$$
[w] \leq [w'] \iff w \leq w' \iff w_1 \leq w_2 
$$
where $[w],[w'] \in W_1\backslash W / W_2$ and $w_1$ and 
$w_2$ are the maximal 
elements, $w_1 \in [w]$ and $w_2 \in [w']$.
This partial order is geometric in the following sense; 
if $O_1$ and $O_2$ are two $P_I$-orbits in $G/P_J$
with the corresponding double cosets $[w_1]$ and $[w_2]$,
respectively, then $O_1 \subseteq \overline{O_2}$
if and only if $w_1 \leq w_2$. 
The bar on $O_2$ stands for the Zariski closure in $G/P_J$.

\vspace{.25cm}

Now let $[w]$ be a double coset from ${}^{I}{W} \cap W^J$
represented by an element $w\in W$ such that 
$\ell(w) \leq \ell(v)$ for every $v\in [w]$. 
It turns out that the set of all such minimal length double coset 
representatives
is given by ${}^{I}{W} \cap W^J$, where $^{I}{W}$ 
stands for the set of minimal length coset representatives for $W_I \backslash W$. 
We denote $^{I}{W} \cap W^J$ by 
$X_{I,J}^-$. Set $H= I \cap w J w^{-1}$. 
Then $ uw \in W^J$ for $u\in W_I$ 
if and only if $u$ is a minimal length coset 
representative for $W_I/W_H$. 
In particular, every element of $W_I w W_J$ 
has a unique expression of the form 
$uwv$ with $u\in W_I$ is a minimal length 
coset representative of $W_I/W_H$, $v\in W_J$ and 
$\ell(uwv) = \ell(u)+\ell(w)+\ell(v)$.

\vspace{.25cm}

Another characterization of the sets $X_{I,J}^-$ is as follows. 
For $w\in W$, the {\em right ascent set} is defined as 
\begin{align*}
\mt{Asc}_R(w) = \{ s\in R :\ \ell(ws) > \ell(w) \}.
\end{align*}
The {\em right descent set}, $\mt{Des}_R(w)$ is the complement $R-\mt{Asc}_R(w)$. 
Similarly, the {\em left ascent set} of $w$ is 
\begin{align*}
\mt{Asc}_L(w) = \{ s\in R :\ \ell(sw) > \ell(w) \}\ \text{ ($=\mt{Asc}_R(w^{-1})$).}
\end{align*}
Then 
\begin{align}\label{A:another}
X_{I,J}^- &= \{ w\in W:\ I\subseteq \mt{Asc}_L(w)\ \text{ and } J\subseteq \mt{Asc}_R(w) \}\\
&=  \{ w\in W:\ I^c\supseteq \mt{Des}_R(w^{-1})\ \text{ and } J^c\supseteq \mt{Des}_R(w) \}
\end{align}

\vspace{.25cm}

For our purposes we need the distinguished 
set of maximal length representatives for each 
double coset. 
It is given by 
\begin{align}\label{A:another}
X_{I,J}^+ &= \{ w\in W:\ I\subseteq \mt{Des}_R(w^{-1})\ \text{ and } J\subseteq \mt{Des}_R(w) \}\\
&= \{ w\in W:\ I^c\supseteq \mt{Asc}_R(w^{-1})\ \text{ and } J^c\supseteq \mt{Asc}_R(w) \}
\end{align}
For a proof of this characterization of $X_{I,J}^+$, see~\cite[Theorem 1.2(i)]{Curtis85}. 
\begin{Remark}\label{R:opposite}
The Bruhat-Chevalley orders on $X_{I,J}^-$ and $X_{I,J}^+$ are isomorphic. 
\end{Remark}

\subsection{}\label{SS:4}

We mentioned in the introductory section
that Littelmann classified the pairs of  
parabolic subgroups $(P_I,P_J)$ 
corresponding to fundamental dominant weights 
such that the diagonal $G$ action on $G/P_I\times G/P_J$ is spherical.
Said differently, we know all pairs $(I,J)$ of subsets of $R$ 
such that 
\begin{itemize} 
\item $|I| = |J|= |R|-1$, and 
\item $G/P_I\times G/P_J$ is a spherical double flag variety for $G$.
\end{itemize} 
In particular, under the maximality assumption 
of the subsets $I$ and $J$, the poset of 
$G$-orbit closures is a chain, 
see~\cite[Proposition 3.2]{Littelmann}. 
In the light of our Lemma~\ref{L:poset isom}, 
this is equivalent to the statement that 
with respect to the Bruhat-Chevalley order, the set 
$X_{I,J}^+$ is a chain.

\vspace{.5cm}

We mention also that the classification of 
Littelmann is extended by Stembridge to 
cover all pairs of subsets $(I,J)$ in $R$ 
such that $G/P_I\times G/P_J$ is a spherical double flag variety for $G$. 
See Corollaries 1.3.A -- 1.3.D, 1.3.E6, 1.3.E7, 
and 1.3.\{E8,F4,G2\} in~\cite{Stembridge}.

\begin{Remark}\label{R:in order}
\begin{enumerate}
\item We call a spherical double flag variety $G/P_I\times G/P_J$ trivial if 
one of the factors is isomorphic to a point, that is $P_I=G$ or $P_J=G$. 
In the cases of E8,F4, and G2 all of the spherical double flag varieties are trivial.
\item In the cases of A--D, E6, and E7, if $G/P_I\times G/P_J$ is a spherical 
double flag variety for $G$,
then at least one of the subsets $I$ and $J$ is maximal, 
that is to say, of cardinality $|R|-1$. Without loss of generality
we always choose $I$ to be the maximal one. 
\end{enumerate}
\end{Remark}

\subsection{}\label{SS:5}

In this subsection we will review the 
useful concept of ``tight Bruhat order.''
We maintain our notation from the previous
subsections.

\vspace{.5cm}

One way to define the Bruhat-Chevalley 
order on $W$ is to use the reflection representation
of $W$ as the group of isometries of $\mathbf{E}$.
Let $\langle\ ,\rangle$ denote the $W$-invariant 
inner product on $\mathbf{E}$, 
and let $\theta \in \mathbf{E}$ be a vector such that 
$\langle \theta, \beta \rangle \geq 0$ for all $\beta \in \Phi^+$. 
Such a vector is called dominant. It is indeed 
dominant in the sense of Subsection~\ref{SS:1}.

\vspace{.5cm}

It is well known that the stabilizer of a dominant vector 
is a parabolic subgroup $W_J\subset W$, where  
$J=\{s_\alpha \in R:\ \langle \theta ,\alpha \rangle=0\}$. 
Thus, as a set, the minimal length coset representatives 
$W^J \subset W$ of the quotient $W/W_J$ can be identified 
with the orbit $W \theta$. 
Following Stembridge, we are going to call the orbit map 
$w\mapsto w\cdot \theta$ the evaluation.

\vspace{.5cm}

A proof of the following result can be 
found in~\cite{Stembridge02}.
\begin{Proposition}\label{P:BC1}
Let $\theta \in \mathbf{E}$ be a dominant 
vector with stabilizer $W_J$. 
The evaluation map induces a poset isomorphism 
between the Bruhat-Chevalley order on $W^J$ and the orbit 
$W\theta$ with partial order $\leq_B$ defined by the 
transitive closure of the following relations:
$$
\mu \leq_B s_\beta(\mu) \ \text{ for all $\beta \in \Phi^+$ 
such that $\langle \mu ,\beta \rangle > 0$.}
$$
\end{Proposition}

\vspace{.5cm}

Let $I$ be a subset of $R$ 
and let $\Phi_I \subset \Phi$ denote the root
subsystem corresponding to the parabolic subgroup $W_I$.
We denote by $\Phi_I^+$ the intersection $\Phi^+\cap \Phi_I$. 
If $\theta$ is a dominant vector and its stabilizer subgroup is 
$W_J$ with $J\subset R$, then we define 
\begin{align}\label{A:double quotient}
(W\theta)_I:= \{\mu \in W\theta :\ \langle \mu ,\beta \rangle \geq 0 \ 
\text{ for all $\beta \in \Phi^+_I$} \}.
\end{align}

\vspace{.5cm}

A proof of the following result can be found 
in~\cite[Proposition 1.5]{Stembridge05}.
\begin{Proposition}\label{P:BC2}
Let $I,J\subset R$ be two sets of Coxeter generators 
for $W$ and let $\theta \in \mathbf{E}$ 
be a dominant vector with stabilizer $W_J$. 
Then the evaluation map induces a poset 
isomorphism between the (restriction of) 
Bruhat-Chevalley order on $X_{I,J}^-$ and $(W\theta)_I$ 
with partial order defined by the transitive closure of 
the relations
$$
\mu \leq_B s_\beta(\mu) \ \text{ for all $\beta \in \Phi^+$ 
such that $s_\beta (\mu)\in (W\theta)_I$ and 
$\langle \mu ,\beta \rangle > 0$}.
$$
\end{Proposition}

\vspace{.5cm}

Now we come to the definition of a critical notion for our proof. 
There is a natural partial ordering on the roots defined by 
\begin{align}\label{A:natural}
\nu \preceq \mu \iff \mu-\nu \in \R^+ \Phi^+.
\end{align}
It turns out, when the interpretation of Bruhat-Chevalley 
ordering as given in Proposition~\ref{P:BC1}
is used, there is a natural order reversing implication:
\begin{align}\label{A:implication}
\mu \leq_B \nu \implies \nu \preceq \mu.
\end{align}
If the converse implication also holds, 
then the poset $W\theta$ is called tight. 
More precisely, a subposet $(M,\leq_B)$ 
of the Bruhat-Chevalley order on $(W\theta,\leq_B)$ is called tight
if 
$$
\mu \leq_B \nu \iff \nu \preceq \mu
$$
for all $\nu,\mu$ in $M\subset \mathbf{E}$.

\vspace{.5cm}

In the light of our Remark~\ref{R:in order} part 3, 
we assume that $I\subset R$ is a maximal subset 
of the form $I= R-\{s\}$ for some $s\in R$. 
Also, we assume that there exists a dominant 
$\theta \in \mathbf{E}$ such that 
$W_J$ is its stabilizer subgroup. 
Now, by \cite[Theorem 2.3]{Stembridge05}, 
we see that if $W^J$ is tight, then 
$X_{I,J}^- = X_{R-\{s\},J}^-$ is a chain. 
The list of tight quotients is also given 
in~\cite{Stembridge05}; $(W^J,\leq_B)$ 
is tight if and only if $W$ is of at most rank 2, 
or $J=R$, or one of the following holds: 
\begin{itemize}
\item $W\cong A_n$ and $J^c=\{ s_j \}$ $(1\leq j \leq n )$ or $J^c=\{ s_j,s_{j+1} \}$ $(1\leq j \leq n-1 )$,
\item $W\cong B_n$ and $J^c=\{ s_1 \},\{ s_2 \},\{ s_n \}$, or  $J^c=\{ s_1,s_2 \}$,
\item $W\cong D_n$ and $J^c=\{ s_1 \},\{ s_2 \}$ or $J^c=\{ s_n \}$, 
\item $W\cong E_6$ and $J^c=\{ s_1 \}$ or $J^c=\{ s_6 \}$, 
\item $W\cong E_7$ and $J^c=\{ s_7 \}$, 
\item $W\cong F_4$ and $J^c=\{ s_1 \}$ or $J^c=\{ s_4 \}$, or 
\item $W\cong H_3$ and $J^c=\{ s_1 \}$ or $J^c=\{ s_3 \}$. 
\end{itemize}
Therefore, in these cases (when $I$ is maximal 
and $J$ is as in this list) we know that 
$X_{I,J}^- = X_{R-\{s\},J}^-$ is a chain. 
We finish our preliminaries section by listing 
the remaining cases under the assumption
that $I$ is of the form 
$R-\{s\}$ for some $s\in R$.
\begin{itemize}
\item $W\cong A_n$ 
\begin{enumerate}
\item $I^c\in \{ \{s_2\}, \{s_{n-1}\}\}$ and $J^c=\{s_p,s_q\}$ with $1<p <p+1<q<n$;
\item $I^c \in \{ \{s_1\}, \{s_n\}\}$ and $|J^c| \geq 2$ 
(but $J^c\neq \{ s_j,s_{j+1} \}$ $(1\leq j \leq n-1 )$);
\item $I^c \in \{ \{s_2\},\dots,\{s_{n-1}\}\}$, and $J^c=\{s_1,s_j\}$ or $J^c=\{s_j,s_n\}$ with $2< j < n-1$.
\end{enumerate}
\item $W\cong C_n$
\begin{enumerate}
\item $I^c=\{s_n\}$ and $|J^c|=1$;
\item $I^c=J^c=\{s_1\}$.
\end{enumerate}
\item $W\cong D_n$ ($n\geq 4$)
\begin{enumerate}
\item $I^c=\{s_n\}$ and $J^c= \{s_l,s_i \}$ with $1\leq i \leq n$ and $1\leq l \leq 2$;
\item $I^c\in \{ \{s_1 \},\{s_2\}\}$, and 
$J^c \subsetneq \{s_1,s_2,s_n\}$ or $J^c\subseteq \{s_{n-1},s_n\}$ 
or $J^c=\{s_{n-2}\}$;
\item ($n=4$ case only) $I^c= \{s_1\}$ and $J^c=\{s_2,s_3\}$ 
or $I^c=\{s_2\}$ and $J^c=\{s_1,s_3\}$. 
\end{enumerate}
\item $W\cong E_6$
\begin{enumerate}
\item $I^c\in \{ \{s_1\},\{s_6\}\}$ and $J^c=\{s_1,s_6\}$.
\end{enumerate}
\end{itemize}

\section{Proof of the main result}\label{S:main}

The Weyl group of $(SL(n+1),T)$, where $T$ 
is the maximal torus of diagonal matrices 
is isomorphic to the symmetric group $S_{n+1}$.
A set of Coxeter generators $R \subset S_{n+1}$ is given by 
the set $$R= \{ s_i = (i,\, i+1) :\ i=1,\dots, n\},$$
where $(i,\ i+1)$ is the simple transposition
that interchanges $i$ and $i+1$ and leaves 
everything else fixed. 
For easing our notation, whenever it is clear
from the context, we will denote the simple 
transposition $s_i$ by its index $i$.

\vspace{.25cm}

In the light of Lemma~\ref{L:poset isom},
Subsection~\ref{SS:3}, and Subsection~\ref{SS:5},
to prove our main result Theorem~\ref{T:main},
it will suffice to analyze the Bruhat-Chevalley 
order on the set of distinguished double coset 
representatives, $X_{I,J}^+$. We will do 
this analysis on a case-by-case basis for
\begin{enumerate}
\item $I^c\in \{ \{2\}, \{{n-1}\}\}$ and $J^c=\{p,q\}$ ($1<p <p+1<q<n$);
\item $I^c \in \{ \{1\}, \{n\}\}$ and $|J^c| \geq 2$ 
(but $J^c\neq \{ j, {j+1} \}$ $(1\leq j \leq n-1 )$);
\item $I^c \in \{ \{2\},\dots,\{n-1\}\}$, and $J^c=\{1,j\}$ or $J^c=\{j, n\}$ with $2< j < n-1$.
\end{enumerate}

\subsection{Case 1.}

We start with a general remark which we will use in the sequel.
\begin{Remark}\label{R:can induct}
Let $w\in S_{n+1}$ be a permutation 
whose one-line notation ends with the 
decreasing string $k \ k-1 \dots 2 \ 1$. 
In this case, any element in the upper interval $[w,w_0]\subset S_{n+1}$ 
has the same ending. In other words, 
if $w'\in [w,w_0]$, then the last $k$ entries of $w'$ are
exactly $k,k-1,\dots, 1$ in this order. 
Similarly, if $w$ begins with the decreasing 
string $n+1\ n  \dots \ k$ for some $k\in \{1,\dots, n+1\}$, then 
any element in the upper interval $[w,w_0]\subset S_{n+1}$ 
has the same beginning. So, essentially, these elements 
form an upper interval of $S_{n+1}$, which is isomorphic to 
$S_{n+1-k}$. In a similar way, if we consider the 
set of permutations that starts with the string $1\, 2\,\dots k$,
then we obtain a lower interval that is isomorphic to $S_{n+1-k}$ in $S_{n+1}$.
\end{Remark}

Now we proceed to give our proof, 
starting with the sub-case $I^c = \{2\}$.

Let $w=w_1\dots w_{n+1}$ be an element, 
in one-line notation, from $X_{I,J}^+$. Recall that 
$$
X_{I,J}^+ =\{ w\in W:\ I^c\supseteq \mt{Asc}_R(w^{-1})\ 
\text{ and } J^c\supseteq \mt{Asc}_R(w) \}.
$$
The meaning of $I^c =\{ 2\} \supseteq \mt{Asc}_R(w^{-1})$ 
is that either $ \mt{Asc}_R(w^{-1})=\emptyset$,
in which case $w$ is equal to $w_0$, the longest permutation, 
or, $ \mt{Asc}_R(w^{-1})=\{2\}$ hence 
$2$ comes before $3$ in $w$ and there are 
no other consecutive pairs $(a,a+1)$ such that $a$ comes before 
$a+1$ in $w$. 
Note also that $ \mt{Asc}_R(w)$ cannot be empty 
unless $X_{I,J}^+=\{ w_0\}$. 

\vspace{.5cm}

We continue with the assumption that $w\neq w_0$. 
Suppose $ J^c = \{ p,q \}$ for $1 < p < p+1 < q < n$. 
We are going to write $L_1$ for the segment $w_1 w_2 \dots w_{p}$, 
$L_2$ for the segment 
$w_{p+1} \dots w_{q}$, and $L_3$ for the segment 
$w_{q+1} \dots w_{n+1}$.
By our assumptions, all three of these 
segments are decreasing sequences. 
In particular, since $2$ comes before $3$ in $w$, 
$2$ cannot appear in $L_3$. In fact, $2$ and $3$ cannot 
appear in the same segment.

First, we assume that $p=2$. 
Since any element of $X_{I,J}^+$ has 
descents (at least) at the positions 
$J=\{1,\widehat{2},3,4,\dots, \widehat{q},\dots, {n+1}\}$, 
the bottom element $\tau_0$ is either of the form 
\begin{align}\label{A:1 or 2}
\tau_0=2\ 1\ | n+1 \ n \dots n-q + 3 \ \ n-q +2\, \ n-q +1 \dots 3,
\end{align}
or it is of the form 
\begin{align}\label{A:2 or 1}
\tau_0=n+1 \ n \dots n-q + 4 \ \ 2\ \ 1\ | \ n-q + 3 \ \ n-q +2 \dots 3. 
\end{align}
The bars between numbers indicate the 
possible positions of ascents.
Note that the number of inversions of the 
former permutation is $1+ {n-1 \choose 2}$, 
and the rank of the latter is 
\begin{align*}
f_n(q) &:= \left( \sum_{i=1}^{q-2} n+1-i \right) +  1 + \left( \sum_{i=q+1}^{n} n+1-i \right) \\
&= {n+1 \choose 2} +1 - (n+1 -q) - (n+1- (q-1)).
\end{align*}
which is always greater than the former. 
Therefore, the minimal element $\tau_0$ 
of $X_{I,J}^+$ starts with $2\ 1$ 
(as in~\ref{A:1 or 2}).

\vspace{.5cm}

This element has a single ascent at the $2$-nd position. 
We will analyze the covers of $\tau_0$. 
Since an upward covering in Bruhat-Chevalley order is obtained 
by moving a larger number to the front, $n+1$ of $L_2$ moves 
into $L_1$ and accordingly either 2 or 1 from 
$L_1$ moves into $L_2$.

\vspace{.5cm}

Recall that each double coset $W_I z W_J$ 
is an interval of $W$ in Bruhat-Chevalley 
order and $X_{I,J}^+$ consists 
of maximal elements of these intervals 
(see \cite[Theorem 1.2(ii)]{Curtis85}).
It follows from this critical observation that, 
to obtain a covering of $\tau_0$, $1$ has to move, 
and it becomes the 
last entry of $L_2$. 
In other words, the permutation 
$$
\tau_1= n+1\ 2 \ | \ n \dots n-q + 3 \ \ 1\ | \ n-q +2\, \ n-q +1 \dots 3
$$
is the unique element in $X_{I,J}^+$ that covers $\tau_0$.

\vspace{.5cm}

Next, we analyze the covers of $\tau_1$; 
it has only two possible coverings which are obtained as follows:
1) 2 moves into $L_2$ and $n$ moves into 
$L_1$, 2) $1$ moves into $L_3$ and $n-q+2$ moves into $L_2$.
The resulting elements are 
\begin{align*}
\tau_2 &= n+1\ n \ | \ n-1 \dots n-q + 3 \ \ 2\ \ 1\ |\ n-q +2\, \ n-q +1 \dots 3, \\
\tau_3 &= n+1\ 2 \ | \ n \dots n-q + 3 \ \ n-q +2\ |  \ n-q +1 \dots 3\ 1.
\end{align*} 
It is not difficult to see that each of these 
two elements are covered by the same element, namely 
$$
\tau_4 = n+1\ n \ | \ n-1 \dots n-q + 3 \ \  n-q+2\ \   2\ \  |\ \ n-q +1 \dots 3 \ 1.
$$
Observe that, in $\tau_4$ the only entry that can be moved is 2 and this is possible only if 
the inequality $q \leq n-1$ holds.
This agrees with our assumption on $q$. 
Therefore, there exists a unique cover of $\tau_4$, which is $w_0$.
Note that all that is said above is independent of $n$ as long as $p=2$ and $3 < q < n$.
Hence, our poset is as in Figure~\ref{F:stretched diamond}.
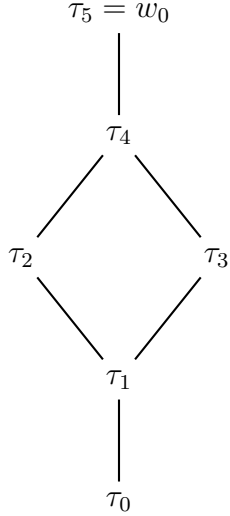
\begin{figure}[htp]
\centering
\begin{tikzpicture}[scale=.65]
\begin{scope}
\node at (0,-5) (t0) {$\tau_0$};
\node at (0,-2.5) (t1) {$\tau_1$};
\node at (-2,0) (t2) {$\tau_2$};
\node at (2,0) (t3) {$\tau_3$};
\node at (0,2.5) (t4) {$\tau_4$};
\node at (0,5) (t5) {$\tau_5=w_0$};
\draw[-, thick] (t0) to (t1);
\draw[-, thick] (t1) to (t2);
\draw[-, thick] (t1) to (t3);
\draw[-, thick] (t2) to (t4);
\draw[-, thick] (t3) to (t4);
\draw[-, thick] (t4) to (t5);
\end{scope}
\end{tikzpicture}
\caption{The Bruhat-Chevalley order on $X_{I,J}^+$ for Case 1.}
\label{F:stretched diamond}
\end{figure}

Finally, we look at the case for $p>2$. 
The only difference between this and $p=2$ case is that 
the first $p-2$ terms of the elements of $X_{I,J}^+$ 
all start with $n+1\ \ n \ \ n-2 \dots n - p$.
By using Remark~\ref{R:can induct} and induction,
we reduce this case to the case of $p=2$.
Therefore, our poset $X_{I,J}^+$ is isomorphic to 
the one in Figure~\ref{F:stretched diamond}.

\vspace{.25cm}

We proceed with the second sub-case of Case 1;
we assume that $I^c= \{n-1\}$ and 
$J=\{ p,q\}$ with $2\leq p< p+1 < q \leq n-1$. 
As in the previous sub-case, for an element 
$w\in X_{I,J}^+$ these conditions imply 
that $w$ is of the form 
$w= L_1 | L_2 | L_3$, where $L_i$, $i=1,2,3$ 
are decreasing sequences of lengths 
$p, q-p$ and $n+1-q$, respectively, 
and the number $n-1$ appears before $n$ in $w$. 
It follows that the smallest element of $X_{I,J}^+$ is of the form 
$$
\tau_0=w_1\dots w_p | w_{p+1}\dots w_q | w_{q+1} \dots w_{n+1} = 
 n-1 \ n-2 \dots n-q  \ | \ n+1\ \ n \ \ n-q -1 \ \ n-q-2 \dots 1
$$
Then arguing exactly as in the previous 
case one sees that the poset under consideration 
is also of the form 
Figure~\ref{F:stretched diamond}.

\subsection{Case 2.}

We start with the sub-case $I^c =\{1\}$, and we let generously 
$J^c$ be any proper subset $J^c \subset \{1,\dots, n\}$.
Let $w= w_1\dots w_{n+1}$ be an element from $X_{I,J}^+$
and let $v_1\dots v_{n+1}$ denote the inverse, $w^{-1}$ of $w$.
Since  $\mt{Asc}_R(w^{-1}) \subseteq \{1\}$, we have 
either $w=w^{-1}=w_0$,
or 
\begin{align}\label{A:obviously}
v_1 < v_2 > v_3 > \cdots > v_{n+1}.
\end{align}
Let $V'$ denote the set of permutations whose entries
satisfy the inequalities in (\ref{A:obviously}) and set 
$$
V:=V'\cup \{w_0\}.
$$
Then $V$ has $n+1$ elements, and furthermore, $(V,\leq)$ 
is a chain. 
But in Bruhat-Chevalley order we have 
$$
u \leq v \iff u^{-1} \leq v^{-1} \ \text{ for every } u,v\in S_{n+1}.
$$
Therefore, 
$V^{-1}:= \{ v^{-1} :\ v\in V\}$ is a chain also.
It follows that, as a subposet of $V^{-1}$, $X_{I,J}^+$ is a chain as well.
This finishes the proof of the first part of Case 2.

Next, we assume that $I^c=\{n\}$ and let 
$w= w_1\dots w_{n+1}\in X_{I,J}^+$. 
If $w^{-1}=v_1\dots v_{n+1}$ denotes the inverse of $w$, then, as before, 
we have either $w=w^{-1}=w_0$,
or 
\begin{align}\label{A:obviously}
v_1 > v_2 \cdots > v_n < v_{n+1}.
\end{align}
By arguing as in the previous paragraph we see that $X_{I,J}^+$ is a chain 
in this case as well, and hence, the proof of Case 2 is finished.

\subsection{Case 3.}

Now, we proceed with the proof of Case 3 but 
since we have symmetry, 
we will consider the case of 
$I^c=\{i\}$ with $2\leq i \leq  n-1$ and $J^c=\{1, j\}$ with $2 < j <n-1$ only. 
Let us note also that as the number $i\in I^c$ grows up to $\lfloor \frac{n+1}{2} \rfloor$ 
we get more freedom to position $i$ and $i+1$ in an element $w\in X_{I,J}^+$;
this makes $X_{I,J}^+$ grow taller as a poset. 
Now we are ready to present the structure
of our poset in detail. 
\vspace{.5cm}

A generic element 
$w=w_1 \dots w_{n+1}$ from $X_{I,J}^+$ is viewed as a concatenation 
of three segments, $w=L_1 L_2 L_3$
where $L_1=w_1$, $L_2=w_2\dots w_j$, 
and $L_3 = w_{j+1}\dots w_{n+1}$. 
The possible ascents are at the 
$1$-st and at the $j$-th positions. 
At the same time, if $w\neq w_0$, 
then we have $w^{-1} \neq w_0$, 
therefore, $w^{-1}$ has an ascent at the $i$-th position. 
This means that $i$ comes before $i+1$ in $w$ and there are 
no other pairs $(a,a+1)$ such that $a$ comes before $a+1$ 
in $w$. 
Therefore, $i$ and $i+1$ are always contained 
in distinct segments except for $w=w_0$. 
In particular, $i$ appears either 
in $L_1$ or in $L_2$. 

\vspace{.25cm}

We proceed to determine the smallest element 
$\tau_0$ of $X_{I,J}^+$. 
Let us write $\tau_0$ in the form $\tau_0=L_1L_2L_3$ 
as in the previous paragraph
and let $k$ be the number in $L_1$. 
We observe that if $k\neq n+1$, then we have $k=i$.
Indeed, if we assume otherwise that $k\neq i$ and that $k\neq n+1$,
then we find that $k+1$ comes after $k$ in $\tau_0$; this is a contradiction. 
As a consequence of this observation we see that 
$\tau_0$ starts either with $n+1$ or with $i$. 
On the other hand, if $k=n+1$, then by interchanging 
$k$ with the first entry of $L_2$ we obtain another element in $X_{I,J}^+$
and this new element is smaller than $\tau_0$ in Bruhat-Chevalley order.
This is a contradiction as well. 
Therefore, in $\tau_0$, we have $i$ as the first entry.   
Now there are two easy cases;

1) $j \leq i$ and $\tau_0$ is of the form 
\begin{align}\label{A:bottom element 1}
\tau_0= i \ \ i-1 \dots i-j+1 \ | \ n+1 \ \ n \dots i+1 \ \ i- j \ \ i-j -1 \dots 1.
\end{align}

2) $j > i$ and $\tau_0$ is of the form 
\begin{align}\label{A:bottom element 2}
\tau_0= i \ \ n+1 \ \ n \dots n+2 - (j-i) \ \ i-1 \ \ i-2	\dots 1 \ | \ n+1 - (j-i) \ \ n - (j-i) \dots i+1.
\end{align}
Note that the vertical bar is between the $j$-th and the $j+1$-st positions.

\vspace{.25cm}

We proceed with some observations regarding 
how the posets climb up in 
the Bruhat-Chevalley order on $X_{I,J}^+$,
starting with $\tau_0$'s as in (\ref{A:bottom element 1}) and (\ref{A:bottom element 2}). 
First of all, if $\tau_0$ is as in (\ref{A:bottom element 1}), then to get a covering 
relation, there is only one possible interchange,  
namely, moving $i-j+1 \in L_2$ into $L_3$. 
In this case, to maintain the descents, 
the number that is replaced by 
$i-j+1$ has to be $n+1$, 
which goes into the first entry of $L_2$. 
In other words, the unique $w\in X_{I,J}^+$
that covers $\tau_0$ is 
\begin{align}\label{A:bottom element 3}
w= i \ \ n+1 \ \ i-1 \dots i-j +2 \ | \ n \dots i+1 \ \ i- j+1 \ \  i- j \ \ i-j -1 \dots 1.
\end{align}
It is easy to verify that there are exactly 
two elements that covers $w$;
\begin{align}\label{A:bottom element 4}
w_{(2)}= n+1 \ \ i \ \ i-1 \dots i-j +2 \ | \ n \dots i+1 \ \ i- j+1 \ \  i- j \ \ i-j -1 \dots 1
\end{align}
and
\begin{align}\label{A:bottom element 5}
w^{(2)}= i \ \ n+1 \ \ n \ \ i-1 \dots i-j+3 \ | \  n -1 \dots i+1 \ \ i- j+2 \ \  i- j +1  \dots 1.
\end{align}
By Remark~\ref{R:can induct} 
we see that all elements that lie 
above $w_{(2)}$ in $X_{I,J}^+$
start with $n+1$. Also, since there is 
no ascent at the $1$-st position for such elements, 
the resulting upper interval $[w_{(2)},w_0]$ in 
$X_{I,J}^+$ is isomorphic to a double coset 
poset in $S_{n+1}$ with $I^c=\{i\}$ and $J^c=\{j\}$, 
hence it is a chain.

\vspace{.5cm}

There are two covers of $w^{(2)}$; 
one of them, $w_{(3)}$, is an element of the interval $[w',w_0]$ 
(hence $w_{(3)}$ covers $w'$ as well). 
The other cover of $w^{(2)}$ is 
\begin{align}\label{A:bottom element 6}
w^{(3)}= i \ \ n+1 \ \ n \ \ n-1 \ \ i-1 \dots i-j+4 \ | \  n -1 \dots i+1 \ \ i- j+3 \ \  i- j +2  \dots 1.
\end{align}
Now the pattern is clear; 
$w^{(3)}$ has exactly two covers 
one of which lies in $[w_{(3)},w_0]$ 
and the other $w^{(4)}$ has a similar structure. 
Therefore, the bottom portion of the resulting poset is a `ladder', 
as depicted in Figure~\ref{F:Anii2}, 
and the chains $w^{(p)}$ and $w_{(p)}$, 
$p\geq 3$ climb up to meet for the first time either at $w_0$, 
or at    
\begin{align}\label{A:bottom element 7}
w^{(m+1)}=w_{(m+1)}=  n+1 \ \ n\dots  n+1 - (j-2) \ \ i \ | \  n+1 - (j-3) \dots \widehat{i} \dots 2 \ \ 1.
\end{align}
In the latter case, of course, $w_0$ is the unique cover 
of $w^{(m+1)}=w_{(m+1)}$ and 
it is easy to check from (\ref{A:bottom element 7}) that 
this happens if and only if $n+1-(j-1) > i$.
In both of these cases, the hight of our poset does not exceed $j$.

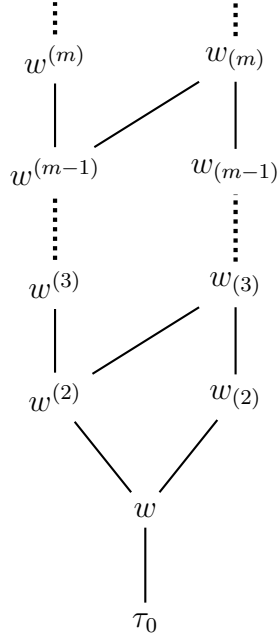
\begin{figure}[htp]
\centering
\begin{tikzpicture}[scale=.6]

\begin{scope}
\node at (0,-7.5) (t0) {$\tau_0$};
\node at (0,-5) (t1) {$w$};
\node at (-2,-2.5) (t2) {$w^{(2)}$};
\node at (-2,0) (t4) {$w^{(3)}$};
\node at (2,-2.5) (t3) {$w_{(2)}$};
\node at (2,0) (t5) {$w_{(3)}$};
\node at (-2,2.5) (t6) {$w^{(m-1)}$};
\node at (2,2.5) (t7) {$w_{(m-1)}$};
\node at (-2,5) (t8) {$w^{(m)}$};
\node at (2,5) (t9) {$w_{(m)}$};
\node at (-2,6.5) (t10) {};
\node at (2,6.5) (t11) {};
\draw[-, thick] (t0) to (t1);
\draw[-, thick] (t1) to (t2);
\draw[-, thick] (t1) to (t3);
\draw[-, thick] (t2) to (t4);
\draw[-, thick] (t3) to (t5);
\draw[-, thick] (t2) to (t5);
\draw[dotted, ultra thick] (t5) to (t7);
\draw[dotted, ultra thick] (t4) to (t6);
\draw[dotted, ultra thick] (t8) to (t10);
\draw[dotted, ultra thick] (t9) to (t11);
\draw[-, thick] (t6) to (t8);
\draw[-, thick] (t7) to (t9);
\draw[-, thick] (t6) to (t9);
\end{scope}
\end{tikzpicture}
\caption{The Bruhat-Chevalley order on $X_{I,J}^+$ for $I^c=\{i\}$, $J^c=\{1,j\}$,
where $2< j \leq i$.}
\label{F:Anii2}
\end{figure}

Now we look at the covers of $\tau_0$ in the case of (\ref{A:bottom element 2}). 
In this case, there are exactly two covers of $\tau_0$: 
\begin{align}\label{A:bottom element 21}
w_{(1)}= n+1 \ \ i \ \ n \dots n+2 - (j-i) \ \ i-1 \ \ i-2 \dots 1 \ | \ n+1 - (j-i) \ \ n - (j-i) \dots i+1
\end{align}
and
\begin{align}\label{A:bottom element 22}
w^{(1)}= i \ \ n+1 \ \ n \dots n+2 - (j-i) \ \ n+1 - (j-i) \ \ i-1 \ \ i-2 \dots 2 \ | \ \ n - (j-i) \dots i+1 \ \ 1.
\end{align}
The elements of $X_{I,J}^+$ that cover 
(\ref{A:bottom element 21}) and (\ref{A:bottom element 22}) are found in a similar 
way to those of (\ref{A:bottom element 4}) and 
(\ref{A:bottom element 5}). We depict the bottom portion of $X_{I,J}^+$
for $\tau_0$ as in (\ref{A:bottom element 2}) in Figure~\ref{F:Anii22}.
Note that, as in the previous case, 
the chains of the poset climb up to meet for the first time either at $w_0$, 
or at    
\begin{align*}
w^{(m+1)}=w_{(m+1)}=  n+1 \ \ n\dots  n+1 - (j-2) \ \ i \ | \  n+1 - (j-3) \dots \widehat{i} \dots 2 \ \ 1.
\end{align*}
The latter situation occurs if and only if $n+1-(j-1) > i$, or, equivalently, $n > i+j -2$.

\begin{figure}[htp]
\centering
\begin{tikzpicture}[scale=.6]

\begin{scope}
\node at (0,-5) (t1) {$\tau_0$};
\node at (-2,-2.5) (t2) {$w^{(1)}$};
\node at (-2,0) (t4) {$w^{(2)}$};
\node at (2,-2.5) (t3) {$w_{(1)}$};
\node at (2,0) (t5) {$w_{(2)}$};
\node at (-2,2.5) (t6) {$w^{(m-1)}$};
\node at (2,2.5) (t7) {$w_{(m-1)}$};
\node at (-2,5) (t8) {$w^{(m)}$};
\node at (2,5) (t9) {$w_{(m)}$};
\node at (-2,6.5) (t10) {};
\node at (2,6.5) (t11) {};
\draw[-, thick] (t1) to (t2);
\draw[-, thick] (t1) to (t3);
\draw[-, thick] (t2) to (t4);
\draw[-, thick] (t3) to (t5);
\draw[-, thick] (t2) to (t5);
\draw[dotted, ultra thick] (t5) to (t7);
\draw[dotted, ultra thick] (t4) to (t6);
\draw[dotted, ultra thick] (t8) to (t10);
\draw[dotted, ultra thick] (t9) to (t11);
\draw[-, thick] (t6) to (t8);
\draw[-, thick] (t7) to (t9);
\draw[-, thick] (t6) to (t9);
\end{scope}
\end{tikzpicture}
\caption{The Bruhat-Chevalley order on $X_{I,J}^+$ for $I^c=\{i\}$, $J^c=\{1,j\}$,
where $2\leq i < j$.}
\label{F:Anii22}
\end{figure}
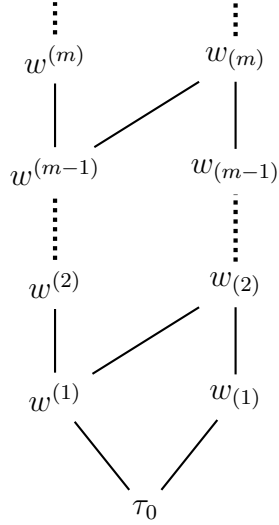

This finishes the proof of Case 3.
By combining with Stembridge's results on 
tight Bruhat order and the classification of the 
spherical double flag varieties for $SL(n+1)$, 
we now have a proof of the following result.

\begin{Theorem}\label{T:Type A G orbits}
Let $G$ denote $SL(n+1)$ and let $P_I$ and $P_J$ 
be two standard parabolic subgroups of $G$. If $G/P_I \times G/P_J$
is a spherical double flag variety, then the inclusion poset $(Z,\subseteq)$ of $G$-orbit closures 
is either a chain or one of the ``ladder lattices'' as depicted in Figure~\ref{F:ladder}.
More precisely, we have  

\begin{enumerate}

\item if $|I^c|=|J^c|=1$, then $Z$ is isomorphic to a chain;

\item if $|I^c|=1$ and $J^c = \{ s_j, s_{j+1}\}$ ($1\leq j \leq n-1$), 
then $Z$ is isomorphic to a chain;

\item if 
$I^c\in \{ \{s_2\}, \{s_{n-1}\}\}$ and $J^c=\{s_p,s_q\}$ ($1<p <p+1<q<n$),
then the Hasse diagram of $Z$ is as 
in Figure~\ref{F:stretched diamond}; 

\item if 
$I^c \in \{ \{s_1\}, \{s_n\}\}$ and $|J^c| \geq 2$ 
(but $J^c\neq \{ s_j,s_{j+1} \}$ $(1\leq j \leq n-1 )$),
then $Z$ is isomorphic to a chain;

\item if
$I^c \in \{ \{s_2\},\dots,\{s_{n-1}\}\}$, and $J^c=\{s_1,s_j\}$ or $J^c=\{s_j,s_n\}$ with $2< j < n-1$,
then   
\begin{enumerate}
\item the Hasse diagram of $Z$ is as 
in (A) in Figure~\ref{F:ladder} for $2 < j \leq i$ and $i+j-2 < n$; 
\item the Hasse diagram of $Z$ is as 
in (B) in Figure~\ref{F:ladder} for $2 < j \leq i$ and $i+j-2 \geq n$; 
\item the Hasse diagram of $Z$ is as 
in (C) in Figure~\ref{F:ladder} for $j > i \geq 2$ and $i+j-2 < n$; 
\item the Hasse diagram of $Z$ is as 
in (D) in Figure~\ref{F:ladder} for $j > i \geq 2$ and $i+j-2 \geq n$.
\end{enumerate}
\end{enumerate}

\end{Theorem}

\begin{figure}[htp]
\centering
\begin{tikzpicture}[scale=.6]

\begin{scope}[xshift= -9.5cm]
\node at (0,-9) (t00) {(A)};
\node at (0,-7.5) (t0) {$\tau_0$};
\node at (0,-5) (t1) {$w$};
\node at (-2,-2.5) (t2) {$w^{(2)}$};
\node at (-2,0) (t4) {$w^{(3)}$};
\node at (2,-2.5) (t3) {$w_{(2)}$};
\node at (2,0) (t5) {$w_{(3)}$};
\node at (-2,2.5) (t6) {$w^{(m-1)}$};
\node at (2,2.5) (t7) {$w_{(m-1)}$};
\node at (-2,5) (t8) {$w^{(m)}$};
\node at (2,5) (t9) {$w_{(m)}$};
\node at (0,7.5) (t10) {$w_{(m+1)}$};
\node at (0,10) (t11) {$w_0$};
\draw[-, thick] (t0) to (t1);
\draw[-, thick] (t1) to (t2);
\draw[-, thick] (t1) to (t3);
\draw[-, thick] (t2) to (t4);
\draw[-, thick] (t3) to (t5);
\draw[-, thick] (t2) to (t5);
\draw[dotted, ultra thick] (t5) to (t7);
\draw[dotted, ultra thick] (t4) to (t6);
\draw[-, thick] (t6) to (t8);
\draw[-, thick] (t7) to (t9);
\draw[-, thick] (t6) to (t9);
\draw[-, thick] (t8) to (t10);
\draw[-, thick] (t9) to (t10); 
\draw[-, thick] (t10) to (t11);
\end{scope}

\begin{scope}[xshift=-3.1cm]
\node at (0,-9) (t00) {(B)};
\node at (0,-7.5) (t0) {$\tau_0$};
\node at (0,-5) (t1) {$w$};
\node at (-2,-2.5) (t2) {$w^{(2)}$};
\node at (-2,0) (t4) {$w^{(3)}$};
\node at (2,-2.5) (t3) {$w_{(2)}$};
\node at (2,0) (t5) {$w_{(3)}$};
\node at (-2,2.5) (t6) {$w^{(m-1)}$};
\node at (2,2.5) (t7) {$w_{(m-1)}$};
\node at (-2,5) (t8) {$w^{(m)}$};
\node at (2,5) (t9) {$w_{(m)}$};
\node at (0,7.5) (t10) {$w_{0}$};
\draw[-, thick] (t0) to (t1);
\draw[-, thick] (t1) to (t2);
\draw[-, thick] (t1) to (t3);
\draw[-, thick] (t2) to (t4);
\draw[-, thick] (t3) to (t5);
\draw[-, thick] (t2) to (t5);
\draw[dotted, ultra thick] (t5) to (t7);
\draw[dotted, ultra thick] (t4) to (t6);
\draw[-, thick] (t6) to (t8);
\draw[-, thick] (t7) to (t9);
\draw[-, thick] (t6) to (t9);
\draw[-, thick] (t8) to (t10);
\draw[-, thick] (t9) to (t10); 
\end{scope}

\begin{scope}[xshift=3.1cm]
\node at (0,-9) (t00) {(C)};
\node at (0,-5) (t1) {$\tau_0$};
\node at (-2,-2.5) (t2) {$w^{(2)}$};
\node at (-2,0) (t4) {$w^{(3)}$};
\node at (2,-2.5) (t3) {$w_{(2)}$};
\node at (2,0) (t5) {$w_{(3)}$};
\node at (-2,2.5) (t6) {$w^{(m-1)}$};
\node at (2,2.5) (t7) {$w_{(m-1)}$};
\node at (-2,5) (t8) {$w^{(m)}$};
\node at (2,5) (t9) {$w_{(m)}$};
\node at (0,7.5) (t10) {$w_{(m+1)}$};
\node at (0,10) (t11) {$w_0$};
\draw[-, thick] (t1) to (t2);
\draw[-, thick] (t1) to (t3);
\draw[-, thick] (t2) to (t4);
\draw[-, thick] (t3) to (t5);
\draw[-, thick] (t2) to (t5);
\draw[dotted, ultra thick] (t5) to (t7);
\draw[dotted, ultra thick] (t4) to (t6);
\draw[-, thick] (t6) to (t8);
\draw[-, thick] (t7) to (t9);
\draw[-, thick] (t6) to (t9);
\draw[-, thick] (t8) to (t10);
\draw[-, thick] (t9) to (t10); 
\draw[-, thick] (t10) to (t11);
\end{scope}

\begin{scope}[xshift=9.5cm]
\node at (0,-9) (t00) {(D)};
\node at (0,-5) (t1) {$\tau_0$};
\node at (-2,-2.5) (t2) {$w^{(2)}$};
\node at (-2,0) (t4) {$w^{(3)}$};
\node at (2,-2.5) (t3) {$w_{(2)}$};
\node at (2,0) (t5) {$w_{(3)}$};
\node at (-2,2.5) (t6) {$w^{(m-1)}$};
\node at (2,2.5) (t7) {$w_{(m-1)}$};
\node at (-2,5) (t8) {$w^{(m)}$};
\node at (2,5) (t9) {$w_{(m)}$};
\node at (0,7.5) (t10) {$w_{0}$};
\draw[-, thick] (t1) to (t2);
\draw[-, thick] (t1) to (t3);
\draw[-, thick] (t2) to (t4);
\draw[-, thick] (t3) to (t5);
\draw[-, thick] (t2) to (t5);
\draw[dotted, ultra thick] (t5) to (t7);
\draw[dotted, ultra thick] (t4) to (t6);
\draw[-, thick] (t6) to (t8);
\draw[-, thick] (t7) to (t9);
\draw[-, thick] (t6) to (t9);
\draw[-, thick] (t8) to (t10);
\draw[-, thick] (t9) to (t10); 
\end{scope}
\end{tikzpicture}
\caption{The ladder posets.}
\label{F:ladder}
\end{figure}
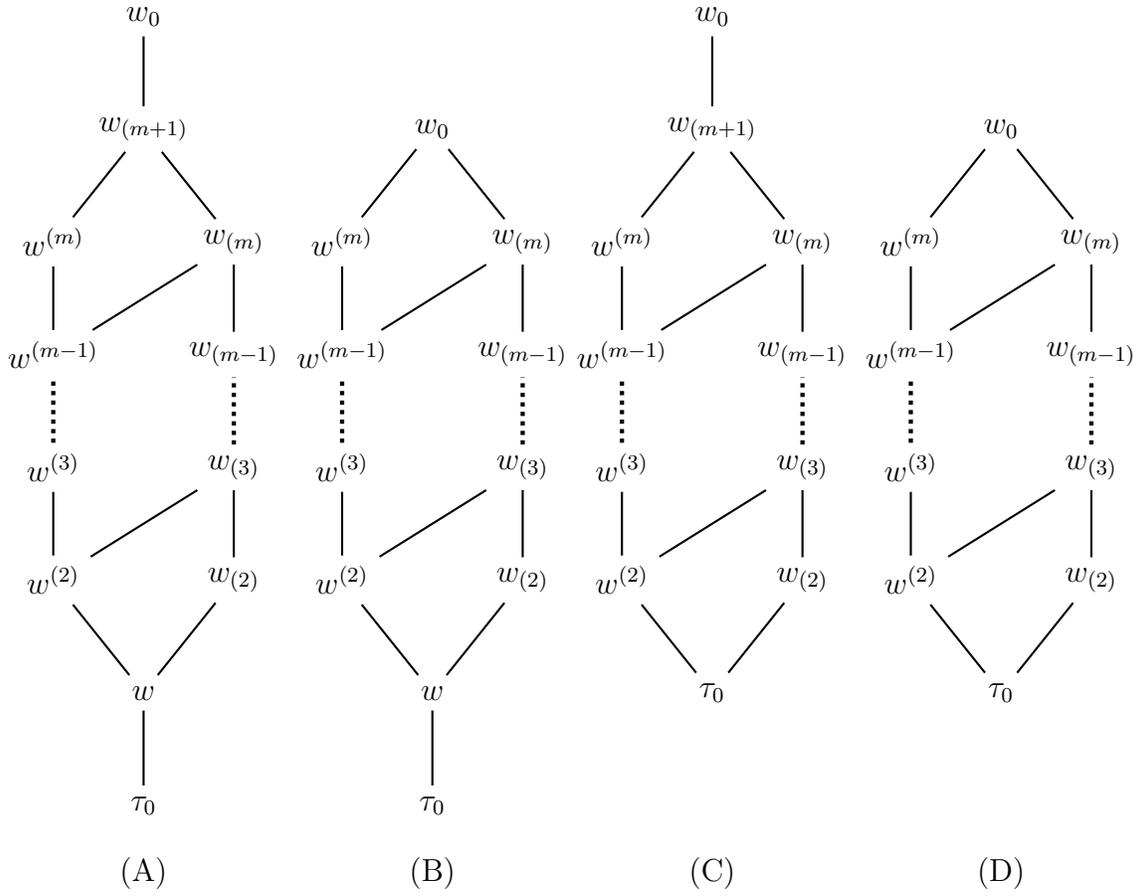

\vspace{.5cm}

\textbf{Acknowledgements.} 
We are grateful to John Stembridge for several reasons, 
including for his Maple codes and 
for answering our questions about his work. 
We thank Bill Graham for bringing this problem to our attention.
We thank Roman Avdeev for his comments on the first version of this manuscript. 
Finally, we thank the referee for very careful reading of our paper and for the constructive suggestions.

\bibliography{References}

\begin{thebibliography}{10}

\bibitem{AP}
R.~S. Avdeev and A.~V. Petukhov.
\newblock Spherical actions on flag varieties.
\newblock {\em Mat. Sb.}, 205(9):4--48, 2014.

\bibitem{Borel}
A.~Borel.
\newblock {\em Linear algebraic groups}, volume 126 of {\em Graduate Texts in
  Mathematics}.
\newblock Springer-Verlag, New York, second edition, 1991.

\bibitem{Curtis85}
C.~W. Curtis.
\newblock On {L}usztig's isomorphism theorem for {H}ecke algebras.
\newblock {\em J. Algebra}, 92(2):348--365, 1985.

\bibitem{Heetal}
X.~He, K.~Nishiyama, H.~Ochiai, and Y.~Oshima.
\newblock On orbits in double flag varieties for symmetric pairs.
\newblock {\em Transform. Groups}, 18(4):1091--1136, 2013.

\bibitem{HohlwegSkandera}
C.~Hohlweg and Skandera M.
\newblock A note on {B}ruhat order and double coset representatives.
\newblock https://arxiv.org/abs/math/0511611, 2005.

\bibitem{Littelmann}
P.~Littelmann.
\newblock On spherical double cones.
\newblock {\em J. Algebra}, 166(1):142--157, 1994.

\bibitem{MWZ1}
P.~Magyar, J.~Weyman, and A.~Zelevinsky.
\newblock Multiple flag varieties of finite type.
\newblock {\em Adv. Math.}, 141(1):97--118, 1999.

\bibitem{MWZ2}
P.~Magyar, J.~Weyman, and A.~Zelevinsky.
\newblock Symplectic multiple flag varieties of finite type.
\newblock {\em J. Algebra}, 230(1):245--265, 2000.

\bibitem{Niemann}
B.~Niemann.
\newblock Spherical affine cones for maximal reductive subgroups in exceptional
  cases.
\newblock {\em PhD. Thesis, Universit{\"a}t zu K{\"o}ln}, 2013.

\bibitem{Panyushev99}
D.~I. Panyushev.
\newblock Complexity and rank of actions in invariant theory.
\newblock {\em J. Math. Sci. (New York)}, 95(1):1925--1985, 1999.
\newblock Algebraic geometry, 8.

\bibitem{Stembridge02}
J.~R. Stembridge.
\newblock A weighted enumeration of maximal chains in the {B}ruhat order.
\newblock {\em J. Algebraic Combin.}, 15(3):291--301, 2002.

\bibitem{Stembridge}
J.~R. Stembridge.
\newblock Multiplicity-free products and restrictions of {W}eyl characters.
\newblock {\em Represent. Theory}, 7:404--439, 2003.

\bibitem{Stembridge05}
J.~R. Stembridge.
\newblock Tight quotients and double quotients in the {B}ruhat order.
\newblock {\em Electron. J. Combin.}, 11(2):Research Paper 14, 41, 2004/06.

\bibitem{VanPruijssen}
M.~van Pruijssen.
\newblock Multiplicity free induced representations and orthogonal polynomials.
\newblock {\em International Math Research Notices}, 2017.
\newblock DOI: 10.1093/imrn/rnw295.

\bibitem{PopovVinberg}
\`E.~B. Vinberg and V.~L. Popov.
\newblock A certain class of quasihomogeneous affine varieties.
\newblock {\em Izv. Akad. Nauk SSSR Ser. Mat.}, 36:749--764, 1972.

\end{thebibliography}
\bibliographystyle{plain}

\end{document}